\documentclass[12pt]{article}
\usepackage{mathrsfs}
\usepackage{amsfonts,amsthm, amssymb}
\usepackage[tbtags]{amsmath}
\usepackage{latexsym, euscript, epic, eepic}
\usepackage{graphicx}
\usepackage{epstopdf}
\usepackage{lineno}

\numberwithin{equation}{section}
\newtheorem{cro}{Corollary}[section]
\newtheorem{defn}{Definition}[section]

\newtheorem{thm}{Theorem}[section]
\newtheorem{lem}{Lemma}[section]

\begin{document}

\title{Large Deviation Properties in Some
Nonuniformly Hyperbolic Systems Via Pesin Theory
 \footnotetext {* Corresponding author}
  }
\author{Zheng Yin$^{1}$, Ercai Chen$^{*1,2}$ \\
  \small   1 School of Mathematical Sciences and Institute of Mathematics, Nanjing Normal University,\\
   \small   Nanjing 210023, Jiangsu, P.R.China\\
    \small 2 Center of Nonlinear Science, Nanjing University,\\
     \small   Nanjing 210093, Jiangsu, P.R.China.\\
      \small    e-mail: zhengyinmail@126.com ecchen@njnu.edu.cn  \\
}
\date{}
\maketitle

\begin{center}
 \begin{minipage}{120mm}
{\small {\bf Abstract.} This article is devoted to level-1
large deviation properties in some nonuniformly
hyperbolic systems via Pesin theory. In particular, our result can
be applied to the nonuniformly hyperbolic diffeomorphisms described by Katok
and several other classes of diffeomorphisms derived from Anosov systems. }
\end{minipage}
 \end{center}

\vskip0.5cm
\section{Introduction}
The classical theory of large deviations deals with
the rates of convergence of sequence of random variables to some
limit distribution. We refer to Ellis' book \cite{Ell} for a general theory
of large deviations and its background in statistical mechanics.
We consider large deviations for dynamical systems. Let $(M,d,f)$ be a topological
dynamical system and let $\mu$ be a ergodic measure. Given an observable $\varphi\in L^1(\mu)$, let
$S_n\varphi(x)=\sum_{i=0}^{n-1}\varphi(f^ix)$. By Birkhoff ergodic theorem,
$(1/n)S_n\varphi(x)\to\int\varphi d\mu$ as $n\to\infty$ for almost all $x\in M$.
We are interested in the rate of convergence of this time average.
More generally, given a reference measure $m$, one interesting question is to
study the asymptotic behaviour  of
\begin{align*}
m\left\{x\in M:\frac{1}{n}S_n\varphi(x)\in E\right\}
\end{align*}
for some $E\subseteq \mathbb{R}$. This topic had contributions
by many authors in the recent years. We refer the reader to
\cite{AraPac,Chu,ComRiv,EizKifWei,Kif,MelNic,PfiSul,PolSha,ReyYou,Var,You,Yur} and the references therein.

The large deviation principle can be divided into two steps, i.e., to obtain the upper bound of
\begin{align*}
\limsup\limits_{n\to\infty}\frac{1}{n}\log m\left\{x\in M:\frac{1}{n}S_n\varphi(x)\in E\right\}
\end{align*}
and the lower bound of
\begin{align}\label{equ1.1}
\liminf\limits_{n\to\infty}\frac{1}{n}\log m\left\{x\in M:\frac{1}{n}S_n\varphi(x)\in E\right\}.
\end{align}
The upper bound is usually easier to establish than the lower bound.
An effective method in dynamical systems to obtain the lower estimate for
(\ref{equ1.1}) is ``orbit-gluing approach''. For example, Young \cite{You} formulated quite general
large deviation for dynamical systems via specification property. Pfister and Sullivan \cite{PfiSul} introduced
approximate product property to study the large deviations estimates for $\beta$-shifts.
Yamamoto \cite{Yam} considered the large deviation principle for group automorphisms which
satisfy the almost specification property.
Varandas \cite{Var} introduced a measure theoretical non-uniform specification to obtain large deviations
estimstes for weak Gibbs measures.

Recently, many efforts have been made to extend the theory of large deviations to the scope of
nonuniformly hyperbolic systmes. Ara\'{u}jo and Pac\'{i}fico \cite{AraPac} gave lower as well
as upper bounds in the noninvertible context. Melbourne and Nicol \cite{MelNic}
gave optimal bounds in the Young tower situation (both invertible and noninvertible).
Large dviations principles were also obtained by Yuri \cite{Yur} in the context
shifts with countably many symbols and by Comman and Rivera \cite{ComRiv} for nonuniformly expanding
rational maps. More recently, Chung \cite{Chu} obtained a sufficient condition to hold a full large deviation
principle for Markov tower maps induced from return time functions and also
gave counterexamples when [6]'s conditions do not hold.

Motivated by the work of Young \cite{You} and Varandas \cite{Var} our purpose here is to
give a contribution for large deviations estimates in some nonuniformly hyperbolic systems.
The main tool we use is weak shadowing lemma in Pesin theory. Particularly, our result can be applied
(i) to the diffeomorphisms on surfaces,
(ii) to the nonuniformly hyperbolic diffeomorphisms described by Katok
and several other classes of diffeomorphisms derived from Anosov systems.

This article is organized as follows. In section 2, we provide some notions
and results of Pesin theory and state the main result. Section 3 is devoted
to the proof of the main result. Examples and applications are given in section 4.


\section{ Preliminaries}
In this section, we first present some notations to be used in this paper.
Then we introduce some notions and results of
Pesin theory \cite{BarPes2,KatHas,Pol} and state the main results.

We denote by $\mathscr M(M,f)$, $\mathscr M_{\rm inv}(M,f)$
and $\mathscr M_{\rm erg}(M,f)$ the set of all Borel probability measures,
$f$-invariant Borel probability measures and ergodic measures
respectively. It is well known that $\mathscr M(M,f)$ equiped
with weak* topology is a compact metrisable
space. We will denote by $D$ the metric on $\mathscr M(M,f)$.
For an $f$-invariant subset $Z\subset X,$ let
$\mathscr M_{\rm inv}(Z,f)$ denote the subset of $\mathscr M_{\rm inv}(M,f)$
for which the measures $\mu$ satisfy $\mu(Z)=1$ and $\mathscr M_{\rm erg}(Z,f)$
denote those which are ergodic. Denote by $C^0(M)$ the space of continuous
functions from $M$ to $\mathbb{R}$ with the sup norm. For $\varphi\in C^0(M)$
and $n\geq1$ we denote $\sum_{i=0}^{n-1}\varphi(f^ix)$ by $S_n\varphi(x)$.
For every $\epsilon>0$, $n\in \mathbb{N}$ and a point $x\in M$, define
$B_n(x,\epsilon)=\{y\in M:d(f^ix,f^iy)<\epsilon,\forall 0\leq i\leq n-1\}$.

Suppose $M$ is a compact connected boundary-less Riemannian
$n$-dimension manifold and $f:X\to X$ is a $C^{1+\alpha}$
diffeomorphism. Let $\mu\in\mathscr M_{\rm erg}(Z,f)$ and $Df_x$ denote
the tangent map of $f$ at $x\in M.$ We say that
$x\in X$ is a regular point of $f$ if there exist
$\lambda_1(\mu)>\lambda_2(\mu)>\cdots>\lambda_{\phi(\mu)}(\mu)$ and a
decomposition on the tangent space $T_x
M=E_1(x)\oplus\cdots\oplus E_{\phi(\mu)}(x)$ such that
\begin{align*}
\lim\limits_{n\to\infty}\frac{1}{n}\log\|(Df^n_x)u\|=\lambda_j(x),
\end{align*}
where $0\neq u\in E_j(x), 1\leq j\leq \phi(\mu).$ The number $\lambda_j(x)$ and
the space $E_j(x)$ are called the Lyapunov exponents and the eigenspaces of
$f$ at the regular point $x,$ respectively. Oseledets theorem \cite{Ose} say that all
regular points forms a Borel set with total measure. For a regular point
$x\in M$, we define
\begin{align*}
\lambda^+(\mu)=\min\{\lambda_i(\mu)|\lambda_i(\mu)\geq0,1\leq i\leq \phi(\mu)\}
\end{align*}
and
\begin{align*}
\lambda^-(\mu)=\min\{-\lambda_i(\mu)|\lambda_i(\mu)\leq0,1\leq i\leq \phi(\mu)\}.
\end{align*}
We appoint $\min\emptyset=0$. An ergodic measure $\omega$ is hyperbolic if $\lambda^+(\omega)$
and $\lambda^-(\omega)$ are both non-zero.

\begin{defn}
Given $\beta_1,\beta_2\gg\epsilon>0$ and for all $k\in\mathbb{Z}^+,$
the hyperbolic block $\Lambda_k=\Lambda_k(\beta_1,\beta_2,\epsilon)$
consists of all points $x\in M$  such that there exists a
decomposition $T_xM=E_x^s\oplus E_x^u$ satisfying:
\begin{itemize}
  \item $Df^t(E_x^s)=E^s_{f^tx}$ and $Df^t(E_x^u)=E^u_{f^tx};$
  \item $\|Df^n|E^s_{f^tx}\|\leq e^{\epsilon k}e^{-(\beta_1-\epsilon)n}e^{\epsilon|t|},\forall t\in\mathbb{Z},n\geq1;$
  \item $\|Df^{-n}|E^u_{f^tx}\|\leq e^{\epsilon k}e^{-(\beta_2-\epsilon)n}e^{\epsilon|t|},\forall t\in\mathbb{Z},n\geq1;$
  \item $\tan (\angle(E^s_{f^tx},E^u_{f^tx}))\geq e^{-\epsilon k}e^{-\epsilon|t|},\forall t\in\mathbb{Z}.$
\end{itemize}
\end{defn}

\begin{defn}
$
\Lambda=\Lambda(\beta_1,\beta_2,\epsilon)=\bigcup\limits_{k=1}^\infty\Lambda_k(\beta_1,\beta_2,\epsilon)
$ is a Pesin set.
\end{defn}
The following statements are elementary properties of Pesin blocks (see \cite{Pol}):
\begin{itemize}
  \item[(1)] $\Lambda_1\subseteq\Lambda_2\subseteq\cdots;$
  \item[(2)] $f(\Lambda_k)\subseteq\Lambda_{k+1},f^{-1}(\Lambda_k)\subseteq\Lambda_{k+1};$
  \item[(3)] $\Lambda_k$ is compact for each $k\geq1$;
  \item[(4)] For each $k\geq1$, the splitting $\Lambda_k\ni x\mapsto E_x^s\oplus E_x^u$ is continuous.
\end{itemize}
The Pesin set $\Lambda(\beta_1,\beta_2,\epsilon)$ is an
$f$-invariant set but usually not compact. Given an ergodic measure
$\mu\in\mathscr M_{\rm erg}(M,f)$, denote by $\mu|\Lambda_l$ the
conditional measure of $\mu$ on $\Lambda_l.$ Let
$\widetilde{\Lambda}_l=$ supp$(\mu|\Lambda_l)$ and
$\widetilde{\Lambda}_\mu=\bigcup_{l\geq1}\widetilde{\Lambda}_l.$
If $\omega$ is an ergodic hyperbolic measure for $f$ and $\beta_1\leq\lambda^-(\omega)$ and
$\beta_2\leq\lambda^+(\omega)$, then $\omega\in \mathscr M_{\rm inv}(\widetilde{\Lambda}_\omega,f)$.

Let $\{\delta_k\}_{k=1}^\infty $ be a sequence of positive real
numbers. Let $\{x_n\}_{n=-\infty}^\infty$ be a sequence of points in
$\Lambda=\Lambda(\beta_1,\beta_2,\epsilon)$ for which there exists a
sequence $\{s_n\}_{n=-\infty}^{\infty}$ of positive integers satisfying:
\begin{equation*}\begin{split}
&\text{(a) } x_n\in\Lambda_{s_n},\forall n\in\mathbb{Z};\\
&\text{(b) } |s_n-s_{n-1}|\leq 1, \forall n\in\mathbb{Z};\\
&\text{(c) } d(f(x_n),x_{n+1})\leq\delta_{s_n},  \forall n\in\mathbb{Z},
\end{split}\end{equation*}
then we call $\{x_n\}_{n=-\infty}^\infty$ a $\{\delta_k\}_{k=1}^\infty$ pseudo-orbit.  Given $\eta>0$
a point $x\in M$ is an $\eta$-shadowing point for the
$\{\delta_k\}_{k=1}^\infty$ pseudo-orbit  if $d(f^n(x),x_n)\leq
\eta\epsilon_{s_n},\forall n\in\mathbb{Z},$ where
$\epsilon_k=\epsilon_0e^{-\epsilon k}$ and $\epsilon_0$ is a constant only dependent on the
system of $f$.

\vskip0.3cm

\noindent \textbf{Weak shadowing lemma.} \cite{Hir,KatHas,Pol}
{\it
Let $f:M\to M$ be a $C^{1+\alpha}$ diffeomorphism, with a non-empty
Pesin set $\Lambda=\Lambda(\beta_1,\beta_2,\epsilon)$ and fixed
parameters, $\beta_1,\beta_2\gg \epsilon>0.$ For $\eta>0$ there exists
a sequence $\{\delta_k\}$ such that for any $\{\delta_k\}$ pseudo
orbit there exists a unique $\eta$-shadowing point.
}
\vskip0.3cm

Let $m$ be a finite Borel measure on $M$. We think of $m$ as our reference measure.
We define
\begin{align*}
\mathcal{V}^+=\big\{\psi\in C^0(M):\exists C,\epsilon>0
\text{ such that } \forall x\in M \text{ and }& \forall n\geq0, \\
&mB_n(x,\epsilon)\leq Ce^{-S_n\psi(x)}\big\}
\end{align*}
and
\begin{align*}
\mathcal{V}^-=\big\{\psi\in C^0(M):&\exists \text{ arbitrarily small }\epsilon>0
\text{ and } C=C(\epsilon)>0 \\
&\text{ such that } \forall x\in M \text{ and } \forall n\geq0,
mB_n(x,\epsilon)\geq Ce^{-S_n\psi(x)}\big\}.
\end{align*}
For $\varphi\in C^0(M)$ and $E\subseteq \mathbb{R}$, we write
\begin{align*}
\overline{R}(\varphi,E)=
\limsup_{n\to \infty}\frac{1}{n}\log m\left\{x\in M:\frac{1}{n}S_n\varphi(x)\in E\right\}
\end{align*}
and
\begin{align*}
\underline{R}(\varphi,E)=
\liminf_{n\to \infty}\frac{1}{n}\log m\left\{x\in M:\frac{1}{n}S_n\varphi(x)\in E\right\}.
\end{align*}
Now, we state the main result of this paper as follows:
\begin{thm}\label{thm2.2}
Let $f:M\to M$ be a $C^{1+\alpha}$ diffeomorphism of a compact Riemannian manifold, with a non-empty
Pesin set $\Lambda=\Lambda(\beta_1,\beta_2,\epsilon)$ and fixed
parameters, $\beta_1,\beta_2\gg \epsilon>0.$ Then for every $\varphi\in C^0(M)$  and $c\in \mathbb{R}$,
the following hold:
\begin{enumerate}
\item[(1)] For $\psi\in \mathcal{V}^+$, we have
\begin{align*}
\overline{R}(\varphi,[c,\infty))
\leq\sup\left\{h_\nu(f)-\int\psi d\nu:\nu\in \mathscr M_{\rm inv}(M,f),\int\varphi d\nu\geq c\right\}.
\end{align*}
\item[(2)] For $\psi\in \mathcal{V}^-$ and $\mu\in \mathscr{M}_{erg}(M,f)$, we have
\begin{align*}
\underline{R}(\varphi,(c,\infty))
\geq\sup\left\{h_\nu(f)-\int\psi d\nu:\nu\in \mathscr{M}_{inv}(\widetilde{\Lambda}_\mu,f),\int\varphi d\nu> c\right\}.
\end{align*}
\end{enumerate}
\end{thm}
\begin{cro}\label{cro2.1}
Let $f:M\to M$ be a $C^{1+\alpha}$ diffeomorphism of a compact Riemannian manifold
and let $\omega\in\mathscr M_{\rm erg}(M,f)$ be a hyperbolic measure. If
$\beta_1\leq\lambda^-(\omega)$ and $\beta_2\leq\lambda^+(\omega)$,
then for every $\varphi\in C^0(M)$  and $c\in \mathbb{R}$, the following hold:
\begin{enumerate}
\item[(1)] For $\psi\in \mathcal{V}^+$, we have
\begin{align*}
\overline{R}(\varphi,[c,\infty))
\leq\sup\left\{h_\nu(f)-\int\psi d\nu:\nu\in \mathscr M_{\rm inv}(M,f),\int\varphi d\nu\geq c\right\}.
\end{align*}
\item[(2)] For $\psi\in \mathcal{V}^-$, we have
\begin{align*}
\underline{R}(\varphi,(c,\infty))
\geq\sup\left\{h_\nu(f)-\int\psi d\nu:\nu\in \mathscr{M}_{inv}(\widetilde{\Lambda}_\omega,f),\int\varphi d\nu> c\right\},
\end{align*}
\end{enumerate}
where $\widetilde{\Lambda}_\omega=\bigcup_{l\geq1}{\rm supp}(\omega|\Lambda_l(\beta_1,\beta_2,\epsilon))$.
\end{cro}

\section{Proof of Main Theorem}
In this section, we will verify theorem \ref{thm2.2}. The upper bound for the measure
of deviation holds for all topological dynamical systems. By \cite{You}, for $\psi\in \mathcal{V}^+$ we have
\begin{align*}
\overline{R}(\varphi,[c,\infty))
\leq\sup\left\{h_\nu(f)-\int\psi d\nu:\nu\in \mathscr M_{\rm inv}(M,f),\int\varphi d\nu\geq c\right\}.
\end{align*}
To obtain the lower bound estimate we need to construct a suitable pseudo-orbit. Our method is inspired by
\cite{LiaLiaSUnTia} and \cite{You}.
Firstly, we give some important lemmas as follows.
\begin{lem}\label{lem3.1}{\rm \cite{Kat2}}
Let $(X,d)$ be a compact metric space, $f:X\to X$ be a continuous map and $\mu$
be an ergodic invariant measure. For $\epsilon>0$, $\gamma\in (0,1)$ let $N^\mu(n,\epsilon,\gamma)$ denote the minimum
number of $\epsilon$-Bowen balls $B_n(x,\epsilon)$, which cover a set of $\mu$-measure at least $1-\gamma$. Then
\begin{align*}
h_\mu(f)=\lim_{\epsilon\to0}\liminf_{n\to\infty}\frac{1}{n}\log N^\mu(n,\epsilon,\gamma)=
\lim_{\epsilon\to0}\limsup_{n\to\infty}\frac{1}{n}\log N^\mu(n,\epsilon,\gamma).
\end{align*}
\end{lem}
\begin{lem}{\rm \cite{Boc}}\label{lem3.2}
Let $f:M\to M$ be a $C^{1}$ diffeomorphism of a compact Riemannian manifold and $\mu\in\mathscr M_{\rm inv}(M,f)$.
Let $\Gamma\subseteq M$ be a measurable set with $\mu(\Gamma)>0$ and let
\begin{align*}
\Omega=\bigcup_{n\in \mathbb{Z}}f^n(\Gamma).
\end{align*}
Take $\gamma>0$. Then there exists a measurable function $N_0:\Omega\to \mathbb{N}$ such that for a.e.$x\in\Omega$
and every $t\in[0,1]$ there is some $l\in\{0,1,\cdots,n\}$ such that $f^l(x)\in\Gamma$ and
$\left|(l/n)-t\right|<\gamma$.
\end{lem}

Fix $\psi\in \mathcal{V}^-$. Pick an arbitrary
$\nu\in \mathscr{M}_{inv}(\widetilde{\Lambda},f)$ with $\int\varphi d\nu> c$. We
will prove that
\begin{align*}
\liminf_{n\to \infty}\frac{1}{n}\log\left\{x\in M:\frac{1}{n}S_n\varphi(x)\in E\right\}
\geq h_\nu(f)-\int\psi d\nu-10\gamma-\gamma h_{top}(f).
\end{align*}
for any preassigned $0<\gamma<1$. Let $\delta=\frac{1}{5}(\int\varphi d\nu-c)$.
\begin{lem}\label{lem3.3}
For $\gamma>0,\delta>0$ and $\nu\in\mathscr{M}_{inv}(\widetilde{\Lambda},f)$
there exists a finite convex combination of ergodic probability measures with
rational coefficients $\lambda=\sum\limits_{i=1}^{k}a_{i}\mu_{i}$ such that
\begin{align*}
\mu_{i}(\widetilde{\Lambda})=1,
\int\varphi d\lambda\geq c+4\delta,\int\psi d\nu\geq\int\psi d\lambda-\gamma \text{ and } h_\lambda(f)\geq h_\nu(f)-\gamma.
\end{align*}
\end{lem}
\begin{proof}
We choose $\epsilon_1>0$ sufficiently small such that for any
$\mu_1,\mu_2\in\mathscr M_{\rm inv}(M,f)$ with $D(\mu_1,\mu_2)<\epsilon_1$
we have
\begin{align*}
\left|\int\varphi d\mu_1-\int\varphi d\mu_2\right|<\delta,
\left|\int\psi d\mu_1-\int\psi d\mu_2\right|<\gamma.
\end{align*}
By ergodic decomposition theorem,
we can write $\nu=\int_{\mathscr{M}_{erg}(\widetilde{\Lambda},f)}md\tau(m)$.
Let $\mathscr{P}=\{P_1,P_2,\cdots,P_k\}$ be a partition of $\mathscr{M}_{erg}(\widetilde{\Lambda},f)$ with
diam$\mathscr{P}<\epsilon_1$. Let $a_i=\tau(P_i)$. For every $1\leq i\leq k$ we can choose an
ergodic measure $\mu_i\in P_i$ satisfying $h_{\mu_i}(f)\geq h_m(f)-\gamma$ for $\tau$-almost every $m\in P_i$.
Let $\lambda=\sum\limits_{i=1}^{k}a_{i}\mu_{i}$. Obviously $\mu_i(\widetilde{\Lambda})=1,1\leq i\leq k$.
By ergodic decomposition theorem one can easily prove the remaining inqualities. Since the $a_i$'s can be
approximated by rational numbers, the desired results follows.
\end{proof}
Since $\mu_i(\widetilde{\Lambda})=1$, we can choose $l\in \mathbb{N}$ such that $\mu_i(\widetilde{\Lambda}_l)>1-\gamma$
for all $1\leq i\leq k$.
By lemma \ref{lem3.1}, we can choose $\epsilon'$ sufficiently small so
\begin{align}\label{equ3.1}
d(x,y)<\epsilon'\Rightarrow\left|\varphi(x)-\varphi(y)\right|<\delta,\left|\psi(x)-\psi(y)\right|<\gamma
\end{align}
and
\begin{align*}
\liminf_{n\to\infty}\frac{1}{n}\log N^{\mu_{i}}(n,4\epsilon',\gamma)>h_{\mu_{i}}(f)-\gamma
\end{align*}
for $1\leq i\leq k$. Let $\eta=\frac{\epsilon'}{\epsilon_0},$
it follows from weak shadowing lemma that there is a sequence of numbers $\{\delta_k\}.$
Let $\xi$ be a finite partition of $X$ with diam$(\xi)<\frac{\delta_{l}}{3}$
and $\xi\geq\{\widetilde{\Lambda}_{l}, M\setminus \widetilde{\Lambda}_{l} \}$.
For $1\leq i\leq k, n\in\mathbb{N},$ we consider the set
\begin{align*}
\Lambda_1^n(\mu_{i})=\left\{x\in\widetilde{\Lambda}_{l}:f^q(x)\in\xi(x){\rm ~for~some~} q\in[n,(1+\gamma)n]\right\}
\end{align*}
and
\begin{equation}\label{equ3.2}
\begin{split}
\Lambda_2^n(\mu_{i})
=\bigg\{x\in\widetilde{\Lambda}_{l}:&S_m\psi(x)\leq m(\int \psi d\mu_i+\gamma) \text{ and }\\
&S_m\varphi(x)\geq m(\int \varphi d\mu_{i}-\delta){\rm~for~all~} m\geq n\bigg\},
\end{split}
\end{equation}
where $\xi(x)$ is the element in $\xi$ containing $x.$  Let $\Lambda^n(\mu_{i})=\Lambda_1^n(\mu_{i})\cap \Lambda_2^n(\mu_{i})$.
Then we have the following lemma.
\begin{lem}
$\lim\limits_{n\to\infty}\mu_{i}(\Lambda^n(\mu_{i}))=\mu_{i}(\widetilde{\Lambda}_{l}),1\leq i\leq k.$
\end{lem}
\begin{proof}
By Birkhoff ergodic theorem, it suffices to prove that
\begin{align*}
\lim_{n\to\infty}\mu_{i}(\Lambda_1^n(\mu_{i}))=\mu_{i}(\widetilde{\Lambda}_{l}),1\leq i\leq k.
\end{align*}
Let $[(1+\gamma)n]$ be the integer part of $(1+\gamma)n$. Take an element $\xi_k$ of $\xi$ and let $\Gamma=\xi_k$.
Applying lemma \ref{lem3.2}, we can deduce that for $a.e.x\in \xi_k$, there exists a measurable function $N_0$ such that for every
$[(1+\gamma)n]\geq N_0(x)$ there is some $q\in\{0,1,\cdots,[(1+\gamma)n]\}$ such that $f^q(x)\in \xi_k$ and
\begin{align*}
\frac{q}{[(1+\gamma)n]}-1>-\frac{\gamma}{1+\gamma}.
\end{align*}
We obtain $q\geq n-\frac{1}{1+\gamma}$. Since $q\in \mathbb{N}$, we have $q\in [n,(1+\gamma)n]$ and the proof is completed.
\end{proof}

By lemma \ref{lem3.1}, we can take sufficiently large $K\in\mathbb{N}$ such that
\begin{align*}
\mu_{i}(\Lambda^n(\mu_{i}))>1-\gamma
\end{align*}
for all $n\geq M$ and $1\leq i\leq k.$

Let
\begin{align*}
Q(\Lambda^n(\mu_{i}),4\epsilon')&=\inf\{\sharp S:S \text{ is } (n,4\epsilon') \text{ spanning set for } \Lambda^n(\mu_{i}) \},\\
P(\Lambda^n(\mu_{i}),4\epsilon')&=\sup\{\sharp S:S \text{ is } (n,4\epsilon') \text{ separated set for } \Lambda^n(\mu_{i}) \}.
\end{align*}
Then for all $n\geq K$ and $1\leq i\leq k$, we have
\begin{align*}
P(\Lambda^n(\mu_{i}),4\epsilon')\geq Q(\Lambda^n(\mu_{i}),4\epsilon')\geq N^{\mu_{i}}(n,\epsilon',\gamma).
\end{align*}
We obtain
\begin{align*}
\liminf_{n\to\infty}\frac{1}{n}\log P(\Lambda^n(\mu_{i}),4\epsilon')
\geq \liminf_{n\to\infty}\frac{1}{n}\log N^{\mu_{i}}(n,4\epsilon',\gamma)
>h_{\mu_{i}}(f)-\gamma.
\end{align*}
Thus we can choose $t\in \mathbb{N}$ large enough such that $\exp(t\gamma)>\sharp \xi$
and
\begin{align*}
\frac{1}{t}\log P(\Lambda^{t}(\mu_{i}),4\epsilon')>h_{\mu_{i}}(f)-2\gamma.
\end{align*}
for $1\leq i\leq k$. Let $S_i$
be a $(t,4\epsilon')$-separated set for $\Lambda^{t}(\mu_{i})$ and
\begin{align*}
\# S_i\geq\exp\left(t(h_{\mu_{i}}(f)-3\gamma)\right).
\end{align*}
For each $q\in[t,(1+\gamma)t],$ let
\begin{align*}
V_q(\mu_i)=\{x\in S_i:f^q(x)\in\xi(x)\}
\end{align*}
and let $n_i$ be the value of $q$ which maximizes $\#V_q(\mu_i).$ Obviously,
$n_i\geq t\geq\frac{n_i}{1+\gamma}\geq n_i(1-\gamma).$
Since $\exp(t\gamma)\geq t\gamma+1,$ we have that
\begin{align*}
\#V_{n_i}(\mu_i)\geq\frac{\#S_i}{t\gamma+1}\geq \exp\left(t(h_{\mu_{i}}(f)-4\gamma)\right).
\end{align*}
Consider the element $A_{i}\in\xi$ such that $\#(V_{n_i}(\mu_i)\cap A_{i})$ is maximal.
Let $W_{i}=V_{n_i}(\mu_i)\cap A_{i}$. It follows that
\begin{align*}
\#W_{i}\geq \frac{1}{\#\xi}\#V_{n_i}(\mu_i)\geq \frac{1}{\#\xi}\exp\left(t(h_{\mu_{i}}(f)-4\gamma)\right).
\end{align*}
Since $e^{t\gamma}>\sharp \xi$, $t\geq n_i(1-\gamma)$ and $\#W_i\geq1$, we have
\begin{align*}
\#W_{i}\geq \exp\left(n_i(1-\gamma)(h_{\mu_{i}}(f)-5\gamma)\right).
\end{align*}
Notice that $A_{i}$ is contained in an open subset $U_i$ with  diam$(U_i)\leq3$diam$(\xi).$
By the ergodicity of $\mu,$ for any two measures $\mu_{i_1},\mu_{i_2}$,
there exists $s=s(\mu_{i_1},\mu_{i_2})\in\mathbb{N}$ and
$y=y(\mu_{i_1},\mu_{i_2})\in U_{i_1}\cap\widetilde{\Lambda}_{l}$  such that
$f^s(y)\in U_{i_2}\cap\widetilde{\Lambda}_{l}.$
Letting $C_{i}=\frac{a_{i}}{n_i},$ we can choose an integer $N$ large enough so that
$NC_{i}$, $1\leq i\leq k$ are integers and
\begin{align*}
N\gamma\geq \sum\limits_{i=1}^{k-1}s(\mu_{i},\mu_{i+1}).
\end{align*}
Let $X=\sum\limits_{i=1}^{k-1}s(\mu_{i},\mu_{i+1})$  and
\begin{align*}
Y_n=n\sum\limits_{i=1}^{k}Nn_iC_i+X=nN+X,n=1,2,\cdots.
\end{align*}
Then we have
\begin{align}\label{equ3.3}
\frac{X}{N}\leq\gamma,  \frac{N}{Y_n}\geq \frac{1}{n+\gamma},\frac{X}{Y_n}\leq \frac{\gamma}{n+\gamma}.
\end{align}
For simplicity of the statement, for $x\in M$ define segments of orbits
\begin{align*}
L_{i}(x)&:=(x,f(x),\cdots,f^{n_i-1}(x)), 1\leq i\leq k,\\
\widehat{L}_{i_1,i_2}(x)&:=(x,f(x),\cdots,f^{s(\mu_{i_1},\mu_{i_2})-1}(x)),1\leq i_1,i_2\leq k.
\end{align*}

\begin{figure}[!ht]
\centering
\includegraphics[scale=0.9]{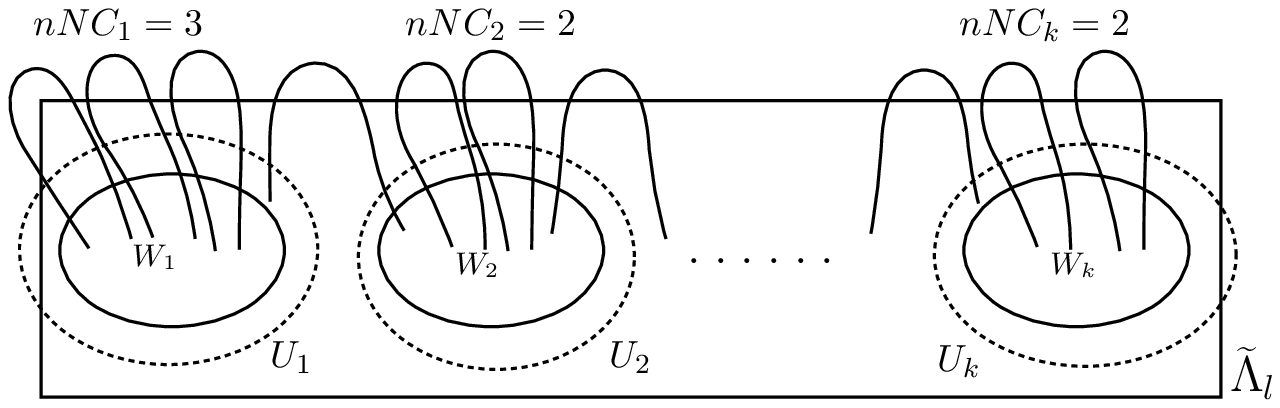}
\caption{Quasi-orbit}\label{figure1}
\end{figure}

For every $n\in \mathbb{N}$, consider the pseudo-orbit with finite length (see Figure \ref{figure1})
\begin{align*}
\Big\{ &L_{1}(x(1,1)),L_{1}(x(1,2)), \cdots,L_{1}(x(1,nNC_{1})),\widehat{L}_{1,2}(y(\mu_{1},\mu_{2}));\\
   &L_{2}(x(2,1)),L_{2}(x(2,2)), \cdots,L_{2}(x(2,nNC_{2})),\widehat{L}_{2,3}(y(\mu_{2},\mu_{3}));\\
   &\vdots\\
   &L_{k}(x(k,1)),L_{k}(x(k,2)), \cdots,L_{k}(x(k,nNC_{k}))
\Big\},
\end{align*}
where $x(i,j)\in W_{i},1\leq i\leq k,1\leq j\leq nNC_i.$

For $1\leq i\leq k$, $1\leq j\leq nNC_i$, let
\begin{align*}
M_1&=0,\\
M_i&=\sum_{l=1}^{i-1}(nNC_ln_l+s(\mu_l,\mu_{l+1})),\\
M_{i,j}&=M_i+(j-1)n_i.
\end{align*}
By weak shadowing lemma, there exists a shadowing $z$ for the above  pseudo-orbit satisfying
\begin{align*}
d(f^{M_{i,j}+q}(z),f^q(x(i,j)))\leq \eta\epsilon_0e^{-\epsilon l}
=\frac{\epsilon'}{\epsilon_0}\epsilon_0e^{-\epsilon l}\leq \epsilon'
\end{align*}
for $0\leq q\leq n_i-1$. The shadowing point $z$ can be considered as a map with the variables $x(i,j)$.
We define $F_n$ to be the image of the map. More precisely,
\begin{align*}
F_n=\Big\{z:z&=z(x(1,1),x(1,2), \cdots,x(1,nNC_{1}),
   x(2,1),x(2,2), \cdots,x(2,nNC_{2}),\cdots, \\
   &x(k,1),x(k,2), \cdots,x(k,nNC_{k})),x(i,j)\in W_i,1\leq i\leq k,1\leq j\leq nNC_i\Big\}.
\end{align*}
Since $W_{i}$ is $(n_i,4\epsilon')$ separated sets, for distinct $x(i,j),x'(i,j)\in W_{i}$
the corresponding shadowing points $z,z'\in F_n$ satisfying $B_{Y_n}(z,\epsilon')\cap B_{Y_n}(z',\epsilon')=\emptyset$.
Thus
\begin{equation}\label{equ3.4}
\begin{split}
\#F_n&\geq(\#W_1)^{nNC_1}(\#W_2)^{nNC_2}\cdots(\#W_k)^{nNC_k}\\
&\geq\exp\left\{\sum_{i=1}^k nNn_iC_i(1-\gamma)(h_{\mu_{i}}(f)-5\gamma)\right\}\\
&=\exp \left\{nN(1-\gamma)(h_\lambda(f)-5\gamma)\right\}.
\end{split}
\end{equation}
\begin{lem}\label{lem3.5}
For sufficiently large $n$, we have $B_{Y_n}(z,\epsilon')\subseteq \left\{x\in M:\frac{1}{Y_n}S_{Y_n}\varphi(x)>c\right\}$
for all $z\in F_n$.
\end{lem}
\begin{proof}
From (\ref{equ3.1}), it suffices to prove that for sufficiently large $n$,
\begin{align*}
\frac{1}{Y_n}S_{Y_n}\varphi(z)>c+\delta.
\end{align*}
By (\ref{equ3.1}), (\ref{equ3.2}), $a_i=C_in_i$ and lemma \ref{lem3.3}, we have
\begin{align*}
\frac{1}{Y_n}S_{Y_n}\varphi(z)
&\geq \frac{1}{Y_n}\sum_{i=1}^{k}\sum_{j=1}^{nNC_i}S_{n_i}\varphi(x(i,j))-\frac{nN\delta}{Y_n}-\frac{X\|\varphi\|}{Y_n}\\
&\geq \frac{1}{Y_n}\sum_{i=1}^{k}nNC_in_i(\int\varphi d\mu_i-\delta)-\frac{nN\delta}{Y_n}-\frac{X\|\varphi\|}{Y_n}\\
&=     \frac{nN}{Y_n}(\int\varphi d\lambda-\delta)-\frac{nN\delta}{Y_n}-\frac{X\|\varphi\|}{Y_n}\\
&\geq  \frac{nN}{Y_n}(c+3\delta)-\frac{nN\delta}{Y_n}-\frac{X\|\varphi\|}{Y_n}.
\end{align*}
It follows from (\ref{equ3.3}) we obtain
\begin{align*}
\lim_{n\to\infty}\frac{nN}{Y_n}=1,\lim_{n\to\infty}\frac{X}{Y_n}=0.
\end{align*}
Thus for sufficiently large $n$, we have
\begin{align}\label{equ3.5}
\frac{1}{Y_n}S_{Y_n}\varphi(z)>c+\frac{3}{2}\delta.
\end{align}
Hence the desired result follows.
\end{proof}
\begin{lem}\label{lem3.6}
For sufficiently large $n$,
\begin{align*}
\frac{1}{Y_n}\log m\left\{x\in M:\frac{1}{Y_n}S_{Y_n}\varphi(x)>c\right\}\geq h_\nu(f)-\int \psi d\nu-10\gamma-\gamma h_{top}(f).
\end{align*}
\end{lem}
\begin{proof}
By lemma \ref{lem3.5}, we have
\begin{align*}
     &\frac{1}{Y_n}\log m\left\{x\in M:\frac{1}{Y_n}S_{Y_n}\varphi(x)>c\right\}\\
\geq &\frac{1}{Y_n}\log\sum_{z\in F_n}mB_{Y_n}(z,\epsilon')
\geq \frac{1}{Y_n}\log\sum_{z\in F_n}C\exp(-S_{Y_n}\psi(z))\\
=    &\frac{1}{Y_n}\log\sum_{z\in F_n}\exp(-S_{Y_n}\psi(z))+\frac{1}{Y_n}\log C.
\end{align*}
Applying (\ref{equ3.1}) and (\ref{equ3.2}), we have
\begin{align*}
S_{Y_n}\psi(z)<&\sum_{i=1}^{k}\sum_{j=1}^{nNC_i}S_{n_i}\psi(x(i,j))+nN\gamma+X\|\psi\|\\
\leq &\sum_{i=1}^{k}nNC_in_i(\int\psi d\mu_i+\gamma)+nN\gamma+X\|\psi\|\\
=    &nN\int\psi d\lambda+2nN\gamma+X\|\psi\|.
\end{align*}
It follows from (\ref{equ3.4}) lemma \ref{lem3.3} that
\begin{align*}
     &\frac{1}{Y_n}\log m\left\{x\in M:\frac{1}{Y_n}S_{Y_n}\varphi(x)>c\right\}\\
\geq &\frac{1}{Y_n}\log \sum_{z\in F_n}\exp\left\{-nN\int\psi d\lambda-2nN\gamma-X\|\psi\|\right\}+\frac{1}{Y_n}\log C\\
\geq &\frac{1}{Y_n}\log \exp\left\{nN(1-\gamma)(h_\lambda(f)-5\gamma)-nN\int\psi d\lambda-2nN\gamma-X\|\psi\|\right\}
        +\frac{1}{Y_n}\log C\\
=&\frac{1}{Y_n}\left\{nN(h_\lambda(f)-\int\psi d\lambda)+nN(5\gamma^2-7\gamma-\gamma h_\lambda(f))-X\|\psi\|\right\}
        +\frac{1}{Y_n}\log C\\
\geq&\frac{1}{Y_n}\left\{nN(h_\nu(f)-\int\psi d\nu)+nN(-9\gamma-\gamma h_\lambda(f))-X\|\psi\|\right\}+\frac{1}{Y_n}\log C.
\end{align*}
Together with
\begin{align*}
\lim_{n\to\infty}\frac{nN}{Y_n}=1,\lim_{n\to\infty}\frac{X}{Y_n}=0,
\end{align*}
for sufficiently large $n$, we have
\begin{align*}
\frac{1}{Y_n}\log m\left\{x\in M:\frac{1}{Y_n}S_{Y_n}\varphi(x)>c\right\}
\geq h_\nu(f)-\int\psi d\nu-10\gamma-\gamma h_{top}(f).
\end{align*}
\end{proof}
For any $n\in \mathbb{N}$, let $i_n$ be the unique natural number such that
\begin{align*}
Y_{i_n}\leq n<Y_{i_n+1}.
\end{align*}
By (\ref{equ3.5}),  for sufficiently large $n$  we have
\begin{align*}
\frac{1}{n}S_{n}\varphi(z)>c+\delta \text{ and }
B_{n}(z,\epsilon')\subseteq \left\{x\in M:\frac{1}{n}S_{n}\varphi(x)>c\right\}
\end{align*}
for all $z\in F_{i_n}$.
Combining this with previous lemma we obtain
\begin{align*}
     &\frac{1}{n}\log m\left\{x\in M:\frac{1}{n}S_{n}\varphi(x)>c\right\}\\
\geq &\frac{Y_{i_n}}{n}\cdot\frac{1}{Y_{i_n}}\log\sum_{z\in F_{i_n}} mB(z,\epsilon')\\
\geq &\frac{Y_{i_n}}{n}\left\{h_\nu(f)-\int\psi d\nu-10\gamma-\gamma h_{top}(f)\right\}.
\end{align*}
for sufficiently large $n$. Since $\gamma$ was arbitrary, the proof of theorem \ref{thm2.2} is completed.

\section{Some Applications}
{\bf Example 1 Diffeomorphisms on surfaces}
Let $f:M\to M$ be a $C^{1+\alpha}$ diffeomorphism with $\dim M=2$ and
$h_{top}(f)>0$, then there exists a hyperbolic measure $m\in\mathscr M_{\rm erg}(M,f)$
with Lyapunov exponents $\lambda_1>0>\lambda_2$ (see \cite{Pol}). If
$\beta_1=|\lambda_2|$ and $\beta_2=\lambda_1$, then for any $\epsilon>0$
such that $\beta_1,\beta_2>\epsilon$, we have $m(\Lambda(\beta_1,\beta_2,\epsilon))=1$.
Let
\begin{align*}
\widetilde{\Lambda}=\bigcup_{k=1}^{\infty}\text{supp}(m|\Lambda(\beta_1,\beta_2,\epsilon)),
\end{align*}
then for every $\varphi\in C^0(M)$  and $c\in \mathbb{R}$,
the following hold:
\begin{enumerate}
\item[(1)] For $\psi\in \mathcal{V}^+$, we have
\begin{align*}
\overline{R}(\varphi,[c,\infty))
\leq\sup\left\{h_\nu(f)-\int\psi d\nu:\nu\in \mathscr M_{\rm inv}(M,f),\int\varphi d\nu\geq c\right\}.
\end{align*}
\item[(2)] For $\psi\in \mathcal{V}^-$, we have
\begin{align*}
\underline{R}(\varphi,(c,\infty))
\geq\sup\left\{h_\nu(f)-\int\psi d\nu:
\mu\in \mathscr M_{\rm inv}(\widetilde{\Lambda},f),\int\varphi d\nu> c\right\}.
\end{align*}
\end{enumerate}

\noindent
{\bf Example 2 Nonuniformly hyperbolic systems} In \cite{Kat1}, Katok described a construction
of a diffeomorphism on the 2-torus $\mathbb{T}^2$ with nonzero Lyapunov exponents, which is not an Anosov map.
Let $f_0$ be a linear automorphism given by the matrix
\begin{align*}
  A=\begin{pmatrix}
   2&1\\
   1&1
  \end{pmatrix}
\end{align*}
with eigenvalues $\lambda^{-1}<1<\lambda$.
$f_0$ has a maximal measure $\mu_1$. Let $D_r$ denote the disk of radius $r$ centered
at (0,0), where $r>0$ is small, and put coordinates $(s_1,s_2)$ on $D_r$ corresponding to
the eigendirections of $A$, i.e, $A(s_1,s_2)=(\lambda s_1,\lambda^{-1}s_2)$.
The map $A$ is the time-1 map of the local flow in $D_r$
generated by the following system of differential equations:
\begin{align*}
\frac{ds_1}{dt}=s_1\log\lambda, \frac{ds_2}{dt}=-s_2\log\lambda.
\end{align*}
The Katok map is obtained from $A$ by slowing down these equations near the origin.
It depends upon a real-valued function $\psi$, which is defined on the unit interval $[0,1]$
and has the following properties:
\begin{enumerate}
\item[(1)] $\psi$ is $C^\infty$ except at 0;
\item[(2)] $\psi(0)=0$ and $\psi(u)=1$ for $u\geq r_0$ where $0<r_0<1$;
\item[(3)] $\psi^{\prime}(u)>0$ for every $0<u<r_0$;
\item[(4)] $\int_0^1\frac{du}{\psi(u)}<\infty$.
\end{enumerate}
Fix sufficiently small numbers $r_0<r_1$ and consider the time-1 map $g$ generated
by the following system of differential equations in $D_{r_1}$:
\begin{align*}
\frac{ds_1}{dt}=s_1\psi(s_1^2+s_2^2)\log\lambda, \frac{ds_2}{dt}=-s_2\psi(s_1^2+s_2^2)\log\lambda.
\end{align*}
The map $f$, given as $f(x)=g(x)$ if $x\in D_{r_1}$ and $f(x)=A(x)$ otherwise, defines a homeomorphism of
torus, which is a $C^\infty$ diffeomorphism everywhere except for the origin. To provide the differentiability
of map $f$, the function $\psi$ must satisfy some extra conditions. Namely, the integral $\int_0^1du/\psi$ must
converge ``very slowly" near the origin. We refer the smoothness to \cite{Kat1}. Here $f$ is contained in the
$C^0$ closure of Anosov diffeomorphisms and even more there is a homeomorphism $\pi:\mathbb{T}^2\to \mathbb{T}^2$
such that $\pi\circ f_0=f\circ \pi$. Let $\nu_0=\pi_\ast\mu_1$.

In \cite{LiaLiaSUnTia}, the authors proved
that there exist $0< \epsilon\ll \beta$ and a neighborhood $U$ of $\nu_0$ in
$\mathscr M_{\rm inv}(\mathbb{T}^2,f)$ such that for any ergodic $\nu\in U$ it holds that
$\nu\in \mathscr M_{\rm inv}(\widetilde{\Lambda}(\beta,\beta,\epsilon),f)$, where
$\widetilde{\Lambda}(\beta,\beta,\epsilon)=\bigcup_{k\geq1}{\rm supp}(\nu_0|\Lambda_k(\beta,\beta,\epsilon))$.

\begin{cro}
For every $\varphi\in C^0(\mathbb{T}^2)$  and $c\in \mathbb{R}$,
the following hold:
\begin{enumerate}
\item[(1)] For $\psi\in \mathcal{V}^+$, we have
\begin{align*}
\overline{R}(\varphi,[c,\infty))
\leq\sup\left\{h_\nu(f)-\int\psi d\nu:\nu\in \mathscr M_{\rm inv}(\mathbb{T}^2,f),\int\varphi d\nu\geq c\right\}.
\end{align*}
\item[(2)] For $\psi\in \mathcal{V}^-$, we have
\begin{align*}
\underline{R}(\varphi,(c,\infty))
\geq\sup\left\{h_\nu(f)-\int\psi d\nu:
\mu\in \mathscr M_{\rm inv}(\widetilde{\Lambda}(\beta,\beta,\epsilon),f),\int\varphi d\nu> c\right\}.
\end{align*}
\end{enumerate}
\end{cro}

In \cite{LiaLiaSUnTia}, the authors also studied the structure of Pesin set $\widetilde{\Lambda}$
for the robustly transitive partially hyperbolic diffeomorphisms described
by Ma\~{n}\'{e} and the robustly transitive non-partially hyperbolic
diffeomorphisms described by Bonatti-Viana.
They showed that for the diffeomorphisms derived from
Anosov systems $\mathscr M_{\rm inv}(\widetilde{\Lambda},f)$
enjoys many members. So our result is applicable to these maps.


\noindent {\bf Acknowledgements.}   The research was supported by
the National Basic Research Program of China
(Grant No. 2013CB834100) and the National Natural
Science Foundation of China (Grant No. 11271191).

\end{document}